\documentclass[a4paper, 10pt]{amsart}
\usepackage[utf8]{inputenc} % set input encoding (not needed with XeLaTeX)
\usepackage{tikz-cd}
\usepackage{latexsym,amssymb,amsfonts, amsthm, amsmath}
\usepackage{graphicx}
\usepackage{amssymb}
\usepackage{amsmath}
\usepackage{caption}
\usepackage{float}
\usepackage{xcolor}

\newtheorem{thm}{Theorem}[section]
\newtheorem{lem}[thm]{Lemma}
\newtheorem{cor}[thm]{Corollary}
\newtheorem{prop}[thm]{Proposition}

\newtheorem{defn}[thm]{Definition}
\newtheorem{exa}[thm]{Example}

\newtheorem{rem}[thm]{Remark}

\title[Orthogonal and oriented Fano planes]{Orthogonal and oriented Fano planes, \\ triangular embeddings of $K_7,$ and geometrical representations of the Frobenius group $F_{21}$}
\author[S. Costa]{Simone Costa}
\address{DICATAM - Sez. Matematica, Universit\`a degli Studi di Brescia, Brescia 25123, Italy}
\email{simone.costa@unibs.it}
\author[M. Pavone]{Marco Pavone}
\address{Dipartimento di Ingegneria, Universit\`a degli Studi di Palermo, Palermo 90128, Italy}
\email{marco.pavone@unipa.it}
\begin{document}

\begin{abstract}
In this paper we present some geometrical representations of the Frobenius group of order $21$ (henceforth, $F_{21}$). The main focus is on investigating the group of common automorphisms of two orthogonal Fano planes and the automorphism group of a suitably oriented Fano plane. We show that both groups are isomorphic to $F_{21},$ independently of the choice of the two orthogonal Fano planes and of the choice of the orientation.

\smallskip
We show, moreover, that any triangular embedding of the complete graph $K_7$ into a surface is isomorphic to the classical toroidal biembedding and hence is face $2$-colorable, with the two color classes defining a pair of orthogonal Fano planes. As a consequence, we show that, for any triangular embedding of $K_7$ into a surface, the group of the automorphisms that preserve the color classes is the Frobenius group of order $21.$

\smallskip
This way we provide three geometrical representations of $F_{21}$. Also, we apply the representation in terms of two orthogonal Fano planes to give an alternative proof that $F_{21}$ is the automorphism group of the Kirkman triple system of order $15$ that is usually denoted as \#61.
\end{abstract}
\maketitle

\section{Introduction}
In this paper we are interested in finding some description of the non-abelian group $F_{21}$ of order $21$ as the group of symmetries of some geometrical object (just like, for instance, $S_3$ is the group of symmetries of the equilateral triangle), or as the group of automorphisms of some combinatorial object (recall, for instance, that the non-abelian group $F_{39}$ of order $39$ is the group of automorphisms of the cyclic STS($13$)). Finding such representations is not in general a surprise, since, because of Frucht's theorem, any group $G$ is the automorphism group of a suitable simple graph $\Gamma$. On the other hand, in the case where $G=F_{21},$ Frucht's theorem requires a quite big graph (with more than $100$ vertices). The goal of this paper is to provide three different (but related) geometrical representations of $F_{21},$ each of which only uses $7$ vertices.

\smallskip
In Section \ref{SecOrthogonal} we show that $F_{21}$ is the group of common automorphisms of two orthogonal Fano planes. Orthogonal STSs were first introduced by O'Shaughnessy \cite{Shaughnessy} as a means of constructing Room squares.
Formally, the notion of a Fano plane and the orthogonality between two Fano planes can be defined as follows.

\begin{defn}\label{definitionsts} \emph{(see, e.g., \cite{Anderson3, BJL, CRC, ColbournRosa}) A \emph{Steiner triple system} of order $v,$ denoted STS($v$), is a pair ($\mathcal{V}, \mathcal{B}$), where $\mathcal{V}$ is a set of $v$ elements (\emph{points}), and $\mathcal{B}$ is a collection of (unordered) triples of elements of $\mathcal{V}$ (\emph{blocks}), with the property that each (unordered) pair of points occurs as a subset of precisely one block in $\mathcal{B}.$ A Steiner triple system of order $7$ is also called \emph{Fano plane}.}

\emph{An \emph{isomorphism} from an STS ($\mathcal{V}_1, \mathcal{B}_1$) to an STS ($\mathcal{V}_2, \mathcal{B}_2$) is a one-to-one map $\sigma$ from $\mathcal{V}_1$ onto $\mathcal{V}_2$ that preserves triples: more precisely, $T=\{x,y,z\} \in \mathcal{B}_1$ if and only if $\sigma(T)=\{\sigma(x),\sigma(y),\sigma(z)\} \in \mathcal{B}_2.$ An \emph{automorphism} is an isomorphism from an STS to itself. The group of all the automorphisms of a Steiner triple system $\mathcal{D} = (\mathcal{V}, \mathcal{B})$ is denoted by $\operatorname{Aut}(\mathcal{D}).$}
%Given a Steiner triple system $\mathcal{D} = (\mathcal{V}, \mathcal{B}),$ we denote by $\operatorname{Aut}(\mathcal{D})$ the group of all the automorphisms of $\mathcal{D}.$}
\end{defn}

\begin{defn}\label{definitionorthogonallsts} \emph{Two STSs ($\mathcal{V},\mathcal{B}_1$) and ($\mathcal{V},\mathcal{B}_2$) on the same point-set are \emph{disjoint} if $\mathcal{B}_1\cap\mathcal{B}_2 = \emptyset,$ that is, if they have no blocks in common. They are said to be \emph{orthogonal} if, in addition, they satisfy the following property:
$$(\{x,y,z\},\{u,v,z\} \in \mathcal{B}_1, \: \{x,y,a\},\{u,v,b\} \in \mathcal{B}_2) \Longrightarrow a \neq b.$$}
\end{defn}
\medskip
In the special case where the order $v$ is equal to $7,$ it is easy to see that two Fano planes are orthogonal if and only if they are disjoint.

\begin{exa}\label{ex1}
\emph{Let us consider the family $\mathcal{\bar{B}}_1$ of the triples that are obtained by translating $\{0,1,3\}$ (mod $7$). More precisely,
$$ \mathcal{\bar{B}}_1=\{\{0, 1, 3\},  \{1, 2, 4\},\{2, 3, 5\},\{3, 4, 6\},  \{4, 5,0\},  \{5, 6,1\}, \{6,0, 2\}\}.$$}

\emph{Similarly, we define $\mathcal{\bar{B}}_2$ as the set of translates (mod $7$) of the triple $\{0,1,5\}$ or, more precisely,
$$ \mathcal{\bar{B}}_2=\{\{ 0, 1, 5 \},\{ 1, 2, 6 \},\{ 2, 3,0 \},\{ 3,4,1\},\{ 4,5,2 \},\{ 5,6,3 \},\{ 6,0,4 \}\}.$$}

\emph{Since $\mathcal{\bar{B}}_1\cap\mathcal{\bar{B}}_2=\emptyset,$ it is immediate that $\bar{\mathcal{F}}_1=(\mathcal{V},\mathcal{\bar{B}}_1)$ and $\bar{\mathcal{F}}_2=(\mathcal{V},\mathcal{\bar{B}}_2)$ are orthogonal Fano planes on the common point-set $\mathcal{V}=\{0,1,2,3,4,5,6\}.$}
\end{exa}

%\begin{defn}
%\emph{Given two orthogonal Fano planes $\mathcal{F}_1$ and $\mathcal{F}_2$, we define the common automorphism group of $\mathcal{F}%_1$ and $\mathcal{F}_2$ as
%$$\operatorname{Aut}(\mathcal{F}_1,\mathcal{F}_2)=\operatorname{Aut}(\mathcal{F}_1)\cap \operatorname{Aut}(\mathcal{F}_2).$$}
%\end{defn}

\begin{exa}\label{lambdatwo}
\emph{If we identify the point-set $\mathcal{V}$ in Example \ref{ex1} with $\mathbb{Z}_7,$ then one can readily verify by inspection that the map $\lambda_2(x) = 2x$ (mod $7$) is a common automorphism of the orthogonal Fano planes $\bar{\mathcal{F}}_1$ and $\bar{\mathcal{F}}_2$.}
\end{exa}

%\smallskip
In the final part of Section \ref{SecOrthogonal}, we apply the main result on the group of common automorphisms of two orthogonal Fano planes
to show that $F_{21}$ is also the automorphism group of a notable Kirkman triple system of order $15.$

\smallskip
Subsequently, in Section \ref{SecOriented}, we investigate the automorphism group of an oriented Fano plane, and show that it is isomorphic to $F_{21}$ as well. Formally, an oriented Fano plane and its automorphism group can be defined as follows.

\begin{defn}\label{oriented}
\emph{A Fano plane $\mathcal{F}=(\mathcal{V},\mathcal{B})$ is \emph{oriented} if there exists a binary antisymmetric relation $\rightarrow$ on $\mathcal{V}$ such that
\begin{itemize}
\item for any triple $\{x_1,x_2,x_3\}\in \mathcal{B},$ either $x_1\rightarrow x_2\rightarrow x_3\rightarrow x_1$ or $x_1\rightarrow x_3\rightarrow x_2\rightarrow x_1$ (in particular, for any distinct $x,y$ in $\mathcal{V}$, either $x\rightarrow y$ or $y \rightarrow x$);
\item for any $x\in \mathcal{V}$, if $\{x,y_1,z_1\}$, $\{x,y_2,z_2\},$ and $\{x,y_3,z_3\}$ are the three triples in $\mathcal{B}$ through $x,$ and if $x\rightarrow y_1$, $x\rightarrow y_2,$ and $x\rightarrow y_3,$ then $\{y_1,y_2,y_3\}\in \mathcal{B}$.
\end{itemize}}
\emph{If this is the case, then we say that the relation  $\rightarrow$ is an \emph{orientation} on $\mathcal{F},$ and that the pair $(\mathcal{F}, \rightarrow)$ is an \emph{oriented Fano plane}. }
\end{defn}

\begin{exa}\label{exoriented}
\emph{Let $\bar{\mathcal{F}}_1=(\mathcal{V},\mathcal{\bar{B}}_1)$ be the Fano plane introduced in Example \ref{ex1}, and let us identify again the point-set $\mathcal{V}=\{0,1,2,3,4,5,6\}$ with $\mathbb{Z}_7.$ We define an orientation $\rightarrow$ on $\bar{\mathcal{F}}_1$ by first defining $0\rightarrow 1\rightarrow 3\rightarrow 0$ on the base triple $\{0, 1, 3\},$ and then extending the orientation by translation (mod $7$) on the other triples. In other words, for any $x$ in $\mathcal{V},$ we let $x\rightarrow x+1\rightarrow x+3\rightarrow x.$}

\smallskip
\emph{The definition is well posed, since any two distinct points in $\mathcal{V}$ belong to exactly one triple in $\mathcal{\bar{B}}_1.$  Note that, equivalently, for any two distinct points $x,y$ in $\mathcal{V},$ $x \rightarrow y$ if and only if $y-x$ is a non-zero quadratic residue (mod $7$), that is, if and only if $y-x \in \{1,2,4\}$ (mod $7$), that is, if and only if $y \in \{x+1,x+2,x+4\}$ (mod $7$). Therefore, for any $x\in \mathcal{V}$, the orientation on the three blocks through $x$ is $$x\rightarrow x+1\rightarrow x+3\rightarrow x$$ $$x\rightarrow x+2\rightarrow x+6\rightarrow x$$ $$x\rightarrow x+4\rightarrow x+5\rightarrow x.$$}

\smallskip
\emph{Now $\{x+1,x+2,x+4\},$ being a translate of the block $\{1,2,4\},$ is a block in $\mathcal{\bar{B}}_1,$  hence it follows, by Definition \ref{oriented}, that $(\bar{\mathcal{F}}_1, \rightarrow)$ is an oriented Fano plane.}

\smallskip
\emph{Note that the orientation $\rightarrow$ on $\bar{\mathcal{F}}_1$ can also be represented graphically as in Figure \ref{orientedfano2}. For the sake of simplicity, we indicate only one arrow mark for each triple, but it is understood that the orientation is extended cyclically to all the $2$-subsets of the triple. For instance, the arrow mark at the base of the triangle indicates the orientation $2\rightarrow 3\rightarrow 5\rightarrow 2$ on the triple $\{2, 3, 5\}.$}
\end{exa}

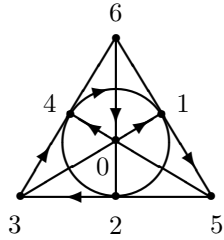
\begin{figure}[H]
% [H] serve a fare in modo che la figura venga messa esattamente nel punto in cui c'e' il testo, anziche' ad esempio ad inizio pagina
\begin{center}
\setlength{\unitlength}{0.19mm}
\begin{picture}(100,160)(0,-30)

\thicklines
\put(24.7,39){\vector(2,3){0}}
\put(35,4){\vector(-1,0){0}}
\put(48.7,56.36){\vector(-3,2){0}}
\put(65,78.7){\vector(4,1){0}}
\put(69.95,55){\vector(0,-1){0}}
\put(93,57.4){\vector(3,2){0}}
\put(125.2,24){\vector(2,-3){0}}

\put(70,43.4){\circle*{5}}
\put(4,4.5){\circle*{5}}
\put(70,4.3){\circle*{5}}
\put(136,4.5){\circle*{5}}
\put(70.2,114.6){\circle*{5}}
\put(38.5,62){\circle*{5}}
\put(101.5,62){\circle*{5}}

\put(70,42){\circle{75}}

\put(70,4){\line(0,1){112}}
\put(3,3){\line(5,3){100}}
\put(3,4){\line(3,5){66}}
\put(137,3){\line(-5,3){100}}
\put(137,4){\line(-3,5){66}}
\put(2,4){\line(1,0){136}}

\put(61,31){\makebox(0,0)[t]{0}}
\put(70,-10){\makebox(0,0)[t]{2}}
\put(0,-9){\makebox(0,0)[t]{3}}
\put(140,-9){\makebox(0,0)[t]{5}}
\put(70,127){\makebox(0,0)[b]{6}}
\put(112,69){\makebox(0,0)[l]{1}}
\put(29,69){\makebox(0,0)[r]{4}}

\end{picture}
\end{center}
\caption{The oriented Fano plane.}
\label{orientedfano2}
\end{figure}

\begin{defn}\label{automoriented}
\emph{Let $(\mathcal{F}, \rightarrow)$ be an oriented Fano plane with point-set $\mathcal{V}$, and let $\sigma$ be an automorphism of $\mathcal{F}$.
We denote by either $\sigma(\rightarrow)$ or $\rightarrow_{\sigma}$ the orientation on $\mathcal{F}$ induced by $\sigma$ and $\rightarrow$ as follows:
for any $x,y$ in  $\mathcal{V},$ $\sigma(x) \rightarrow_{\sigma} \sigma(y)$ if and only if $x \rightarrow y$.
We say that $\sigma$ is an \emph{automorphism} of $(\mathcal{F}, \rightarrow)$ if $\sigma$ preserves the orientation $\rightarrow$, that is, if $\sigma(\rightarrow) = \,\rightarrow$ (equivalently, for any $x,y$ in  $\mathcal{V},$ $\sigma(x) \rightarrow \sigma(y)$ if and only if $x \rightarrow y$).}
\end{defn}

\begin{rem}\label{120degree} \emph{The map $\lambda_2(x) = 2x$ (mod $7$) is an automorphism of the oriented Fano plane $(\bar{\mathcal{F}}_1, \rightarrow)$ defined in Example \ref{exoriented}, since the set $\{1,2,4\}$ of quadratic residues (mod $7$) is invariant under $\lambda_2.$ Note that this map can be visualized in Figure \ref{orientedfano2} as a clockwise $120$-degree rotation of the triangle around its center. Moreover, the map $\tau_1(x)=x+1 \pmod 7$ is an automorphism of $(\bar{\mathcal{F}}_1, \rightarrow)$ as well, by definition of $\rightarrow$.}
\end{rem}

%The main result presented in Section \ref{SecOriented}, dedicated to the oriented Fano plane, is that its automorphism group is again $F_{21}$. %The basic idea of the proof is the construction of a one-to-one correspondence between the family of the orientations on a given Fano plane $\mathcal{F}$ and the family of the Fano planes orthogonal to $\mathcal{F}$.

%\smallskip
The concept of an oriented Fano plane has been originally introduced as a mnemonic rule for the multiplication table of the octonions. In \cite[\S 4.4]{Killgore}, Peter L. Killgore
investigated the group of the automorphisms of the octonions that permute the imaginary basis units $e_1, e_2,e_3,e_4,e_5,e_6,e_7$, and proved that this group is a nonabelian group of order $21$ and is isomorphic to the automorphism group of an oriented Fano plane.
In Section \ref{SecOriented} we give an alternative and more direct proof of this result, by constructing a one-to-one correspondence between the family of the orientations on a given Fano plane $\mathcal{F}$ and the family of the Fano planes orthogonal to $\mathcal{F}.$

\smallskip
In Section \ref{SecEmbeddings}, in order to provide an additional geometrical representation of $F_{21}$, we study the triangular embeddings of the complete graph $K_7$ (i.e., those whose faces have all length three).
\begin{figure}[H]
\centering
\includegraphics[width=240pt,height=160pt]{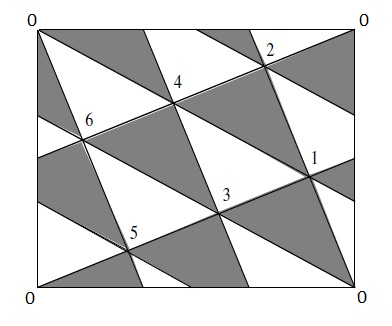}
\caption{A biembedding of $K_7$ into the $2$-torus. The picture is taken from \cite{JS}.}
\label{fig1}
\end{figure}

Here graphs are finite, undirected and simple; loops and multiple edges are not allowed. A \emph{surface} is a compact connected 2-manifold without boundary. Let us recall that a graph embedding into a surface can be defined as follows \cite{Moh, MT}.

\begin{defn}
\emph{Given a graph $\Gamma$ and a surface $\Sigma$, an \emph{embedding} of $\Gamma$ in $\Sigma$ is a continuous injective mapping $\psi: \Gamma \rightarrow \Sigma$, where $\Gamma$ is viewed as a $1$-dimensional simplicial complex, with the usual topology.}
\end{defn}

The connected components of $\Sigma \setminus \psi(\Gamma)$ are called $\psi$-\emph{faces}. If every $\psi$-face is homeomorphic to an open disc, then the embedding $\psi$ is said to be \emph{cellular}. Also, we say that a circuit $T=(x_1,x_2,x_3,\dots,x_{\ell})$ of $\Gamma$ is a \emph{face} of length $\ell$ (induced by the embedding $\psi$) if $\psi([x_1,x_2] \cup [x_2,x_3] \cup \ldots \cup [x_{\ell},x_1])$ is the boundary of a $\psi$-face (where $[x_i,x_{i+1}]$ denotes the arc with endpoints $x_i$ and $x_{i+1}$ in the $1$-dimensional simplicial complex). For instance, the faces of the classical toroidal embedding of $K_7$ in Figure \ref{fig1} (see also \cite{A, GGS, JS}) are precisely the blocks of the orthogonal Fano planes $\bar{\mathcal{F}}_1$ and $\bar{\mathcal{F}}_2$ introduced in Example \ref{ex1}.

\begin{defn}\label{defembiso}
\emph{We say that two graph embeddings $\psi: \Gamma \rightarrow \Sigma$ and $\psi': \Gamma' \rightarrow \Sigma'$ are \emph{isomorphic} whenever there exists a graph isomorphism $\sigma: \Gamma\rightarrow \Gamma'$ such that $\sigma(T)$ is a face in $\Gamma'$ (induced by $\psi'$) if and only if $T$ is a face in $\Gamma$ (induced by $\psi$). The map $\sigma$ will be called an \emph{embedding isomorphism} (\emph{embedding automorphism} in the case where $\psi' =\psi$).}
\end{defn}

Note that, in Definition \ref{defembiso}, $\sigma$ is a graph isomorphism and not a surface homeomorphism.

\smallskip
Let us recall that the classical embedding of $K_7$ into the $2$-torus admits a face $2$-coloring, that is, a coloration of the faces with two colors and with the property that any edge belongs to two faces of different colors. In Figure \ref{fig1} we color in grey the face $\{0,1,3\}$ and its translates (which are the elements of $\mathcal{\bar{B}}_1$ in Example \ref{ex1}), and in white the face $\{0,1,5\}$ and its translates (which are the elements of $\mathcal{\bar{B}}_2$). We also recall that, in general, a graph embedding that admits such a coloring is said to be a \emph{biembedding}. 

\smallskip
We prove that any triangular embedding of $K_7$ into a surface is isomorphic to the classical toroidal one and hence is face $2$-colorable. Moreover, we show that, for any such embedding $\psi,$ the group $\operatorname{Aut}(\psi,Col)$ of the embedding automorphisms that fix the color classes is isomorphic to the automorphism group of two orthogonal Fano planes and hence is $F_{21}$.

\section{Orthogonal Fano Planes}\label{SecOrthogonal}\smallskip
\subsection{Automorphisms of two orthogonal Fano planes} $ $\\

%\smallskip
In this section we aim to determine the group of common automorphisms of two orthogonal Fano planes.

\smallskip
First of all, we introduce some notation and terminology. From now on, we identify the point-set of two orthogonal Fano planes with $\mathbb{Z}_7 = \{0,1,2,3,4,5,6\},$ and we often indicate the triple $\{x,y,z\},$ briefly, by $xyz$ (e.g., $124$ indicates the triple $\{1,2,4\}\subseteq \mathbb{Z}_7$). In the special case where the triple $xyz$ belongs to the family of blocks of a Fano plane $\mathcal{F},$ we also say, for short, that $\mathcal{F}$ contains the triple $xyz,$ or that $xyz$ is a triple in $\mathcal{F}.$

\smallskip
We begin by proving that the group of common automorphisms of two orthogonal Fano planes does not depend on the choice of the two Fano planes.

\begin{prop}\label{p1}
There exist precisely eight orthogonal Fano planes to a given Fano plane.
Moreover, if $\mathcal{F}$ is a Fano plane with point-set $\mathcal{V}$, and if $\mathcal{S}_1,$ $\mathcal{S}_2$ are two distinct Fano planes orthogonal to $\mathcal{F},$ then there exists a permutation of $\mathcal{V}$ that is an automorphism of $\mathcal{F}$ and also an isomorphism between $\mathcal{S}_1$ and $\mathcal{S}_2$.
%if $(\mathcal{F}, \mathcal{S}_1)$ and $(\mathcal{F}, \mathcal{S}_2)$ are two pairs of orthogonal Fano planes with a common point-set $\mathcal{V}$, then there exists a permutation of $\mathcal{V}$ that is an automorphism of $\mathcal{F}$ and also an isomorphism between $\mathcal{S}_1$ and $\mathcal{S}_2$.
\end{prop}

\proof Let $\mathcal{F}$ be a given Fano plane, with point-set $\mathcal{V}$. Up to isomorphism, we may assume that $\mathcal{V} = \mathbb{Z}_7 = \{0,1,2,3,4,5,6\},$ and that the seven blocks in $\mathcal{F}$ are $013, 124, 235, 346,$ $450, 561, 602$ (that is, $\mathcal{F}$ is identified with the Fano plane $\bar{\mathcal{F}}_1$ introduced in Example \ref{ex1}).

\smallskip
Any Fano plane orthogonal to $\mathcal{F}$ contains a triple $01x$, with $x$ in $\{2,4,5,6\}$, and a triple $x3y$, where $y$ can only be one of
the two points different from $0,1,3,x$ and from the third point in the same triple with $x$ and $3$ in $\mathcal{F}$. The choice of $x$ and $y$
determines a unique Fano plane $\mathcal{S}$ orthogonal to $\mathcal{F}$. Indeed, up to permutation, we may assume that $x=5.$ Then $y \in \{4,6\},$ say $y=6$ (resp., $y=4$). In this case, $\mathcal{S}$ contains the triples $01x=015, \, x3y=536,$ and $524$ (resp., $015, \, 534,$ and $526$). Since $\mathcal{F}$ contains $602$ (resp., $124$), $\mathcal{S}$ necessarily contains $604$ and $023$ (resp., $146$ and $123$), and, finally, $126$ and $134$ (resp., $024$ and $036$). Hence $\mathcal{S}$ is uniquely determined by $x$ and $y,$ as claimed.
%For instance, if $x=5$ and $y=6$, then $\mathcal{S}$ contains the triples $015, 536,$ and $524$; since $\mathcal{F}$ contains $602$, $\mathcal{S}$ necessarily contains $064$ and $023$, and, finally, $126$ and $134$ (hence $\mathcal{S}$ is the Fano plane $\bar{\mathcal{F}}_2$ in Example \ref{ex1}). The other cases can be examined similarly.
Since there are four possible choices for $x$ and two possible choices for $y$, it follows that there exist precisely eight orthogonal Fano planes to the given Fano plane $\mathcal{F}$.

\smallskip
Now let $\mathcal{S}_1$ and $\mathcal{S}_2$ be two distinct Fano planes orthogonal to $\mathcal{F}$. If $\mathcal{S}_1$ contains $01x$ and $\mathcal{S}_2$ contains $01x'$, for some $x,x'$ in $\{2,4,5,6\}$, then we may assume that $x'=x$, else we replace $\mathcal{S}_2$ by $\sigma(\mathcal{S}_2)$, where $\sigma$ is the unique automorphism of
$\mathcal{F}$ that fixes $0$ and $1$ and maps $x'$ to $x$. Note that $\sigma(\mathcal{S}_2)$ is orthogonal to $\mathcal{F}$, since $\mathcal{S}_2$ is orthogonal to $\mathcal{F}$
and $\sigma(\mathcal{F})=\mathcal{F}.$ Hence we may assume that $\mathcal{S}_1$ and $\mathcal{S}_2$ are precisely the only two Fano planes that
are orthogonal to $\mathcal{F}$ and contain $01x$. If $\mathcal{S}_1$ contains $x3y$ and $\mathcal{S}_2$ contains $x3y'$ (with $y'$ necessarily different from $y$),
then it is easy to show that the permutation $(0 \:1)(y \:y')$ is an automorphism of $\mathcal{F}$ and an isomorphism between $\mathcal{S}_1$ and $\mathcal{S}_2$. This finally proves our claim.
\endproof

\begin{cor}\label{c1}
The group of common automorphisms of two orthogonal Fano planes does not depend on the choice of the two Fano planes.
\end{cor}

\proof Since the Fano plane is unique up to isomorphism, it suffices to show that if $(\mathcal{F}, \mathcal{S}_1)$ and $(\mathcal{F}, \mathcal{S}_2)$ are two
pairs of orthogonal Fano planes, then the group of common automorphisms of $\mathcal{F}$ and $\mathcal{S}_1$ is (isomorphic to) the group of
common automorphisms of $\mathcal{F}$ and $\mathcal{S}_2$. This is now, in turn, an immediate consequence of Proposition \ref{p1}.
\endproof

\begin{thm}\label{OrthogonalGroup}
The group of common automorphisms of two orthogonal Fano planes is isomorphic to the Frobenius group $F_{21}$ of order $21$.
\end{thm}

\proof By Corollary \ref{c1}, it suffices to consider the two orthogonal Fano planes $\bar{\mathcal{F}}_1$ and $\bar{\mathcal{F}}_2$ defined in Example \ref{ex1} (with point-set $\mathbb{Z}_7 = \{0,1,2,3,4,5,6\}$), whose triples are the translates (mod $7$) of the base blocks $013$ and $015$,
respectively.

\smallskip
Let $G$ be the full group of automorphisms of $\bar{\mathcal{F}}_1$, and let $H$ be the subgroup of the common automorphisms of $\bar{\mathcal{F}}_1$ and $\bar{\mathcal{F}}_2$. We first show that $H$ is a group with $21$ elements. Since $G$ ($\cong \operatorname{GL}(3,2)$) has order $168$, this is equivalent to proving that $H$ has precisely eight distinct left cosets $\sigma H, \sigma \in G$.

\smallskip
For any choice of $\sigma_1, \sigma_2$ in $G,$ $$\sigma_1H = \sigma_2H \Longleftrightarrow \sigma_2^{-1}\sigma_1 \in H \Longleftrightarrow \sigma_2^{-1}\sigma_1(\bar{\mathcal{F}}_2) = \bar{\mathcal{F}}_2 \Longleftrightarrow \sigma_1(\bar{\mathcal{F}}_2) = \sigma_2(\bar{\mathcal{F}}_2).$$

On the other hand, for any $\sigma$ in $G$, $\sigma(\bar{\mathcal{F}}_2)$ is a Fano plane orthogonal to $\bar{\mathcal{F}}_1$, and, conversely, by Proposition \ref{p1}, any Fano plane orthogonal to $\bar{\mathcal{F}}_1$ can be obtained as $\sigma(\bar{\mathcal{F}}_2)$ for some $\sigma$ in $G$. Now, again by Proposition \ref{p1}, there exist exactly eight Fano planes orthogonal to $\bar{\mathcal{F}}_1$, hence $H$ has precisely eight left cosets $\sigma H, \sigma \in G$, as claimed.

\smallskip
Since the Frobenius group of order $21$ is the only non-abelian group of order $21$, we are only left to prove that $H$ is non-abelian. Let $\lambda_2$ and $\tau_1$ be the permutations of $\mathbb{Z}_7$ defined by $\lambda_2(x) = 2x$ and $\tau_1(x) = x+1 \pmod{7}$. Now $\tau_1$ is in $H$ by definition of $\bar{\mathcal{F}}_1$ and $\bar{\mathcal{F}}_2$, and, as we already noted in Example \ref{lambdatwo}, $H$ also contains $\lambda_2$. Since $\lambda_2\tau_1(0) = 2$ and $\tau_1\lambda_2(0) = 1$, $H$ is a non-abelian group, as claimed.
\endproof

\begin{rem}\label{remlambdatau} \emph{It is well known that the Frobenius group of order $21$ is the group of permutations of $\mathbb{Z}_7 = \{0,1,2,3,4,5,6\}$ generated by $x \mapsto 2x$ and $x \mapsto x+1$ (mod $7$). Equivalently, $F_{21}$ is the group of all the affine transformations of the form $x \mapsto ax+b,\, x\in \mathbb{Z}_{7},$ with $b \in \mathbb{Z}_{7}$ and $a \in \{1,2,4\}.$ This follows immediately from the proof of Theorem \ref{OrthogonalGroup}. Indeed, since $\lambda_2$ and $\tau_1$ have order $3$ and $7$, respectively, the order of the group $H'$ generated by $\lambda_2$ and $\tau_1$ is a multiple of $21$. On the other hand, $H'$ is a subgroup of $H$, which has order $21$, hence $H'=H$.}
\end{rem}

\begin{rem} \emph{It would be interesting to investigate the group of common automorphisms of two orthogonal STS($v$)s also for $v>7.$ The smallest order $v>7$ for which there exist two orthogonal Steiner triple systems is $v=13$ \cite{Mullin}, and, for this order, there exists, up to isomorphism, only one pair of orthogonal STSs (see, for instance, \cite{Gibbons})}.

\smallskip
\emph{Such a pair can be constructed as follows. The common point-set is the cyclic group $\mathbb{Z}_{13},$ and the blocks of the first STS (which is known as the \emph{cyclic} STS($13$)) are the twenty-six triples that are cyclically generated (mod $13$) by the two base blocks $\{1,3,9\}$ and $\{2,5,6\}$ (see, e.g., \cite[Theorem 2.11]{ColbournRosa}). Also, the group of automorphisms of the STS is the Frobenius group of all affine transformations of the form $$x \mapsto \varphi_{a,b}(x) = ax+b, \;\mbox{$x\in \mathbb{Z}_{13},$}$$ with $b \in \mathbb{Z}_{13}$ and $a \in \{1,3,9\}$ (which is, up to isomorphism, the unique non-abelian group of order $39$).}

\smallskip
\emph{The blocks of the second STS are the triples $\{-u,-v,-w\}$ (mod $13$), as $\{u,v,w\}$ ranges over the blocks of the first STS. It is easy to check that the two STSs are orthogonal. Also, since the map $x \mapsto -x$ is an isomorphism between the two STSs, the automorphisms of the second STS are all the transformations of the form $x \mapsto -\varphi_{a,b}(-x) = ax-b,$ with $b \in \mathbb{Z}_{13}$ and $a \in \{1,3,9\},$ hence the group of the common automorphisms of the two STSs simply coincides with the automorphism group of the first STS.}
\end{rem}

\smallskip
\subsection{The STS($15$) \#61 and the KTS($15$) \#61} $ $\\ 

Here we apply Theorem \ref{OrthogonalGroup} to show that the Frobenius group $F_{21}$ is also the automorphism group of a notable Steiner triple system of order $15$ and of the corresponding Kirkman triple system.

\smallskip
We first need some preliminary definitions (see, e.g., \cite{BJL, CRC, ColbournRosa}).

\begin{defn}\label{definitionkts}
\emph{Given an STS($v$) ($\mathcal{V}, \mathcal{B}$), with $v$ multiple of $3,$ a \emph{parallel class} is a subcollection of $v/3$ mutually disjoint blocks in $\mathcal{B}$ that partition the point-set $\mathcal{V}.$ When the entire collection $\mathcal{B}$ of blocks can, in turn, be partitioned into parallel classes, such a partition $\mathcal{R}$ is called a \emph{resolution} of the STS, and the STS is said to be \emph{resolvable}. If this is the case, then ($\mathcal{V}, \mathcal{B}, \mathcal{R}$) is called a \emph{Kirkman triple system} of order $v,$ denoted KTS($v$), and ($\mathcal{V}, \mathcal{B}$) is its \emph{underlying} STS.}

\smallskip
\emph{An isomorphism from a KTS ($\mathcal{V}_1, \mathcal{B}_1, \mathcal{R}_1$) to a KTS ($\mathcal{V}_2, \mathcal{B}_2, \mathcal{R}_2$) is an isomorphism $\sigma$ from the STS ($\mathcal{V}_1, \mathcal{B}_1$) to the STS ($\mathcal{V}_2, \mathcal{B}_2$) that, in addition, preserves the parallel classes: for any parallel class $\mathcal{C}$ in $\mathcal{R}_1,$ the set $\{\sigma(T)\, |\, T \in \mathcal{C}\}$ is a parallel class in $\mathcal{R}_2.$ An \emph{automorphism} is an isomorphism from a KTS to itself.}
\end{defn}

%\smallskip
Note that any KTS of order $15$ is a solution of the famous \emph{fifteen schoolgirl problem} \cite[p.~48]{Lady}.

\smallskip
There exist eighty non-isomorphic Steiner triple systems of order $15$ \cite{White}, four of which are resolvable \cite{Cole} (cf. \cite[p.~66]{CRC}, \cite[p.~370]{ColbournRosa}). Three of the four resolvable STS($15$)s underlie two non-isomorphic KTSs, whereas the fourth resolvable STS, usually known as \#61, underlies a unique Kirkman triple system of order $15.$ Moreover, the system \#61 is one of the sixteen STS($15$)s containing a unique Fano plane, and, among these, it is the system with the least number of \emph{Pasch configurations} (see, for instance, \cite[Table 1.29]{CRC}, \cite[Table 5.11]{ColbournRosa}). Finally, the corresponding KTS, which is also denoted by \#61, is related to \emph{Sylvester's problem} of the $15$ schoolgirls: can the $35 \times 13$ (unordered) triples of elements of a $15$-set be grouped into the blocks of $13$ different KTS($15$)s? Denniston's 1974 solution \cite{Denniston} (see also \cite[Example 7.3.2]{Anderson3}, \cite[Example 2.71, p.~66]{CRC}) contains $13$ KTS($15$)s all isomorphic to the KTS \#61.

\smallskip
It is well known that the group of automorphisms of the STS($15$) \#61 is isomorphic to the Frobenius group of order $21.$ This property, however, to the best of our knowledge, is reported in the literature as a mere ``fact'', without an actual proof (the same can be said about the automorphism group of the corresponding KTS). The two main references are \cite{Cole} and \cite{MPR}, where the automorphism group is given as the group generated by two explicit permutations on fifteen symbols, with no further details or references. Here we prove that the group is isomorphic to the group of common automorphisms of two orthogonal Fano planes, and hence is isomorphic, by Theorem \ref{OrthogonalGroup}, to the Frobenius group of order $21.$

\smallskip
\begin{thm}\label{automorphismssts61}
The automorphism group of the STS($15$) \#61 is isomorphic to the Frobenius group of order 21.
\end{thm}
\proof Up to isomorphism, the STS($15$) \#61 is the Steiner triple system $\mathcal{D} = (\mathcal{V}, \mathcal{B})$ with point-set $$\mathcal{V}=\{\infty,0,1,2,3,4,5,6,0',1',2',3',4',5',6'\}$$ and block-set $\mathcal{B}$ consisting of the $35$ triples cyclically generated (mod $7$) by the base blocks $00'\infty ,$ $013,$ $01'6',$ $02'5',$ $03'4'$ (see, for instance, \cite[\S 3, Example 5]{Seven}). Let us set
$$\begin{array}{lcl}
\mathcal{P} & = & \{0,1,2,3,4,5,6\} \\
\mathcal{P}^c & = & \{\infty,0',1',2',3',4',5',6'\}.
\end{array}
$$

\smallskip
Let $\bar{F}_1$ be the Fano plane, contained in $\mathcal{D},$ with point-set $\mathcal{P}$ and block-set consisting of the seven triples cyclically generated (mod $7$) by the base block $013.$ As we mentioned earlier, $\mathcal{D}$ contains a unique Fano plane, hence $\bar{F}_1$ is invariant under any automorphism of $\mathcal{D}.$ It follows, in particular, that $\mathcal{P}$ and $\mathcal{P}^c$ are both invariant as well under any automorphism of $\mathcal{D}.$ Also, let $\bar{F}_2$ be the Fano plane, with point-set $\mathcal{P},$ whose blocks are the seven triples cyclically generated (mod $7$) by the triple $015.$ Then, as we already noted before (see Example \ref{ex1}), $\bar{F}_1$ and $\bar{F}_2$ are orthogonal Fano planes.

\smallskip
By construction, $\mathcal{B}$ does not contain any $3$-subset of $\mathcal{P}^c.$ It follows, by definition of STS, that for any $2$-subset $\{x,y\}$ of $\mathcal{P}^c$ there exists a unique block $axy$ in $\mathcal{B},$ with $a$ in $\mathcal{P}.$ Therefore, for any of the fifty-six $3$-subsets $\{x,y,z\}$ of $\mathcal{P}^c,$ there exists a unique $3$-subset $\{a,b,c\}$ of $\mathcal{P}$ such that $axy,$ $bxz,$ and $cyz$ are blocks in $\mathcal{B}.$ Since $\mathcal{P}$ contains thirty-five $3$-subsets, the map $\{x,y,z\} \mapsto \{a,b,c\}$ is not injective.

\smallskip
%By a trivial cardinality argument, most of the thirty-five $3$-subsets $\{a,b,c\}$ of $\mathcal{P}$ are associated with at least two distinct $3$-sets $\{x,y,z\}.$
One can readily verify, by inspection, that the seven blocks of $\bar{F}_1$ and the seven blocks of $\bar{F}_2$ are the only $3$-subsets $\{a,b,c\}$ of $\mathcal{P}$ that are associated with a unique $3$-subset $\{x,y,z\}$ of $\mathcal{P}^c$ (and that any other $3$-set $\{a,b,c\}$ is associated with exactly two $3$-sets $\{x,y,z\}$). More precisely, each block $\{i,i+1,i+3\}$ in $\bar{F}_1$ is associated only with $\{x,y,z\}=\{(i+2)',(i+4)',(i+5)'\},$ and each block $\{i,i+1,i+5\}$ in $\bar{F}_2$ is associated only with $\{x,y,z\}=\{(i+3)',(i+4)',(i+6)'\}$ (mod $7$). It follows that the family of the fourteen blocks in $\bar{F}_1$ and $\bar{F}_2$ is invariant under any automorphism of $\mathcal{D}.$ Since the same holds for $\bar{F}_1,$ we conclude that $\bar{F}_2$ is invariant under any automorphism of $\mathcal{D}$ as well.

\smallskip
It follows from the properties above that the element $\infty$ is the only point in $\mathcal{P}^c$ that never occurs in a triple $\{x,y,z\}$ associated with a block $\{a,b,c\}$ in either $\bar{F}_1$ or $\bar{F}_2.$ Therefore $\infty$ is invariant under any automorphism of $\mathcal{D}.$ This immediately implies, in turn, that any automorphism of $\mathcal{D}$ is uniquely determined by its restriction to $\mathcal{P}.$ Indeed, for any $n$ in $\mathcal{P},$ and for any automorphism $\sigma$ of $\mathcal{D},$ $\sigma(n)$ is in $\mathcal{P}$ and $\sigma(nn'\infty) = \sigma(n)\sigma(n')\infty,$ hence $\sigma(n')$ necessarily coincides with $\sigma(n)'.$

\smallskip
It follows that the restriction map $\sigma \mapsto \sigma |_{\mathcal{P}}$ is an isomorphism between $\operatorname{Aut}\mathcal{D}$ and a subgroup of the group of common automorphisms of the orthogonal Fano planes $\bar{F}_1$ and $\bar{F}_2.$ Hence $\operatorname{Aut}\mathcal{D}$ has at most $21$ elements. On the other hand, the map $\infty \mapsto \infty,$ $n \mapsto n+1,$ $n' \mapsto (n+1)'$ (mod $7$) is by construction an order-$7$ automorphism of $\mathcal{D},$ and it can be readily verified that the map $\infty \mapsto \infty,$ $n \mapsto 2n,$ $n' \mapsto (2n)'$ (mod $7$) is an order-$3$ automorphism of $\mathcal{D}.$ Therefore $\operatorname{Aut}\mathcal{D}$ has at least $21$ elements, and we may finally conclude that the restriction map $\sigma \mapsto \sigma |_{\mathcal{P}}$ is an isomorphism between $\operatorname{Aut}\mathcal{ D}$ and the group of common automorphisms of $\bar{F}_1$ and $\bar{F}_2,$ which in turn, by Theorem \ref{OrthogonalGroup}, is isomorphic to the Frobenius group of order $21.$

\smallskip
This finally completes the proof of the theorem.
\endproof

%\smallskip
As we mentioned earlier, the STS $\mathcal{D} = (\mathcal{V}, \mathcal{B})$ in the proof of Theorem \ref{automorphismssts61} is a resolvable Steiner triple system of order $15,$ usually known as \#61. The thirty-five blocks of the system can be collected in seven parallel classes, which are generated (mod $7$) by the base parallel class $\{00'\infty, 124, 31'5', 54'6', 62'3'\}$ (see, for instance, \cite[\S 3, Example 5]{Seven}). The corresponding Kirkman triple system, which is also denoted by \#61, has the same automorphisms as $\mathcal{D}.$ Indeed, the Steiner triple system $\mathcal{D}$ only contains seven parallel classes \cite{MPR} (cf. \cite[Table 1.29]{CRC}, \cite[p.~370]{ColbournRosa}), which are necessarily permuted by any automorphism of $\mathcal{D}.$ It follows, by definition, that any automorphism of $\mathcal{D}$ is also an automorphism of the KTS \#61 (cf. \cite[System VII]{Cole}). As a consequence of Theorem \ref{automorphismssts61} above, we obtain the following well-known result (see, e.g., \cite{Cole}), for which no complete formal proof seems to have appeared in print in the literature.

\begin{cor}\label{corkts} The automorphism group of the KTS($15$) \#61 is isomorphic to the Frobenius group of order 21.
\end{cor}

\section{Automorphisms of an Oriented Fano Plane}\label{SecOriented}
\smallskip
The goal of this section is to prove that the automorphism group of an oriented Fano plane is $F_{21}$. This result will be proved by showing that the group under consideration is isomorphic to the automorphism group of two orthogonal Fano planes.

\smallskip
Let us recall that, if $\mathcal{F}$ is a Fano plane, then by saying that $\{a,b,c\}$ is a triple in $\mathcal{F}$ we mean that $\{a,b,c\}$ is a block in $\mathcal{F}.$

\begin{lem}\label{l2}
Let $\mathcal{F}$ and $\mathcal{S}$ be two orthogonal Fano planes with point-set $\mathcal{V}$. Then for any triple $\{a,b,c\}$ in $\mathcal{F}$ there exists a unique
triple $\{x,y,z\}$ in $\mathcal{S}$ such that $\{a,b,c\}$ and $\{x,y,z\}$ are disjoint.

Also, for any $v$ in $\mathcal{V}$, there exist a unique triple $\{a,b,c\}$ in $\mathcal{F}$ and a unique
triple $\{x,y,z\}$ in $\mathcal{S}$ such that $\{v\} \cup \{a,b,c\} \cup \{x,y,z\}$ is a partition of $\mathcal{V}$.
\end{lem}

\proof
If $\{a,b,c\}$ is a triple in $\mathcal{F}$, then $\mathcal{S}$ contains a triple through $a$ and $b$, a triple through $a$ and $c$, a triple through
$b$ and $c$, three additional triples with exactly one element in $\{a,b,c\}$, and a unique triple disjoint from $\{a,b,c\}$.
Now let $v$ be in $\mathcal{V}$. Each of the three triples in $\mathcal{F}$ containing $v$ is disjoint from a (unique) triple in $\mathcal{S},$ not containing $v$,
and each of the three triples in $S$ containing $v$ is disjoint from a (unique) triple in $F,$ not containing $v$. This leaves a
unique pair of disjoint triples $T_1$, $T_2,$ not containing $v$, with $T_1$ in $\mathcal{F}$ and $T_2$ in $\mathcal{S}$.
\endproof

\begin{prop}\label{p2} Let $\mathcal{F}=(\mathcal{V},\mathcal{B})$ be a Fano plane, and let $\rightarrow$ be an orientation on $\mathcal{F}$. For each $v$ in $\mathcal{V}$, let $T(v)$ be the triple of the three points $x,y,z$ in $\mathcal{V}$ such that $x\rightarrow v$, $y\rightarrow v$, $z\rightarrow v$. If $\mathcal{B}_{\rightarrow}$ is the family of the seven triples $T(v)$, as $v$ ranges over $\mathcal{V}$, then $\mathcal{F}_{\rightarrow}=(\mathcal{V},\mathcal{B}_\rightarrow)$ is a Fano plane orthogonal to $\mathcal{F}$.

Moreover, the map $\varphi(\rightarrow) = \mathcal{F}_{\rightarrow}$ defines a bijection between the family of the orientations on $\mathcal{F}$ and the family of the Fano planes orthogonal to $\mathcal{F}$.
\end{prop}

\proof
Suppose, for instance, that $\mathcal{F}$ contains the triples $abc, ade, afg, bdf, beg,$ $cdg,$ $cef$. Let $\rightarrow$ be an orientation on $\mathcal{F}$.
We may assume, up to permutation, that $a\rightarrow b \rightarrow c \rightarrow a$ and $a \rightarrow d \rightarrow e \rightarrow a$. Since $bdf$ is a triple in $\mathcal{F}$, this implies, by Definition \ref{oriented}, that the orientation on the triple $afg$ is $a\rightarrow f\rightarrow g\rightarrow a$. Since $c\rightarrow a$ and $adf$ is not a triple in $\mathcal{F}$, the conditions $c \rightarrow d$ and $c \rightarrow f$ cannot both hold, hence either $d \rightarrow c$ or $f \rightarrow c$. If the two latter conditions were both satisfied, then the orientation on the two triples $cdg$ and $cef$ would be $d \rightarrow c \rightarrow g \rightarrow d$ and $f \rightarrow c \rightarrow e \rightarrow f$, against the hypotheses that $c \rightarrow a$ and $age$ is not a triple in $\mathcal{F}$. Therefore one and only one of the two conditions $d \rightarrow c$ and $f \rightarrow c$ occurs. We assume that $d \rightarrow c$ and $c \rightarrow f$, the other case being examined similarly.

\smallskip
The orientation on the two triples $cdg$ and $cef$ is now $d\rightarrow c \rightarrow g \rightarrow d$ and $c \rightarrow f \rightarrow e \rightarrow c$. Since $f \rightarrow e$ and $f \rightarrow g$ (resp., $g \rightarrow a$ and $g \rightarrow d$), and since $beg$ (resp., $ade$) is a triple in $\mathcal{F}$, forcing $f \rightarrow b$ (resp., $g \rightarrow e$), the orientation on the triple $bdf$ (resp., $beg$) is $f\rightarrow b\rightarrow d \rightarrow f$ (resp., $g \rightarrow e \rightarrow b \rightarrow g$). This finally specifies the orientation on all seven triples in $\mathcal{F}$ (and hence on any pair of distinct elements of $\mathcal{V}$). Also, $\mathcal{B}_{\rightarrow}$ is, by definition, the family $\{T(a), T(b),\dots, T(g)\}$ $=$ $\{ceg, aef, bde, abg,$ $dfg, acd, bcf\}$, hence $\mathcal{F}_{\rightarrow}$ is a Fano plane orthogonal to $\mathcal{F}$.

\smallskip
We now need to prove that the map $\varphi(\rightarrow) = \mathcal{F}_{\rightarrow}$ is a bijection. Let $\mathcal{S}$ be a Fano plane orthogonal to $\mathcal{F}.$ We need to show that there exists a unique orientation $\rightarrow$ on $\mathcal{F}$ such that $\mathcal{F}_{\rightarrow} = \mathcal{S}.$

\smallskip
Let $v$ be a given point in $\mathcal{V}$. Since $\mathcal{F}$ and $\mathcal{S}$ are orthogonal Fano planes, by Lemma \ref{l2} there exist a unique triple $\alpha\beta\gamma$ in $\mathcal{F}$ and a unique triple $xyz$ in $\mathcal{S}$ such that $\{v\} \cup \{\alpha,\beta,\gamma\} \cup \{x,y,z\}$ is a partition of $\mathcal{V}$. Therefore, by Definition \ref{oriented}, if $\mathcal{S}=\mathcal{F}_{\rightarrow}$ for some orientation $\rightarrow$ on $\mathcal{F},$ then the triple $T(v)$ coincides necessarily with $xyz$, hence $x \rightarrow v$, $y \rightarrow v$, $z \rightarrow v$, and $v\rightarrow \alpha$, $v\rightarrow \beta$, $v\rightarrow \gamma$. Since $v$ was an arbitrary point in $\mathcal{V}$, it follows that $\mathcal{S}$ determines the orientation $\rightarrow$ on $\mathcal{F}$ uniquely (that is, the map $\varphi $ is injective).

\smallskip
We are now left to prove that the above definition of the orientation $\rightarrow$ induced by $\mathcal{S}$ is well posed (if this is the case, then $\mathcal{S}=\mathcal{F}_{\rightarrow}$ by construction, that is, $\varphi $ is surjective). Equivalently, we need to show that, as $v$ ranges over $\mathcal{V},$ the partitions $\{v\} \cup \{\alpha,\beta,\gamma\} \cup \{x,y,z\}$ of $\mathcal{V}$ determine an orientation on $\mathcal{F}$ unambiguously. Given a triple $ijk$ in $\mathcal{F}$, it suffices to prove that the choices $v=i$, $v=j$, and $v=k$ produce the same orientation on $ijk.$

%\smallskip
%We claim that $\mathcal{F}_{\rightarrow}$ determines $(\mathcal{F}, \rightarrow)$ uniquely. Let $v$ be a given point in $\mathcal{V}$. Since $\mathcal{F}$ and $\mathcal{F}_{\rightarrow}$ are orthogonal Fano planes, by Lemma \ref{l2} there exist a unique triple $\alpha\beta\gamma$ in $\mathcal{F}$ and a unique triple $xyz$ in $\mathcal{F}_{\rightarrow}$ such that $\{v\} \cup \{\alpha,\beta,\gamma\} \cup \{x,y,z\}$ is a
%partition of $\mathcal{V}$. It follows from Definition \ref{oriented} that the triple $T(v)$ coincides necessarily with $xyz$, hence $x \rightarrow v$, $y \rightarrow v$, $z \rightarrow v$, and $v\rightarrow \alpha$, $v\rightarrow \beta$, $v\rightarrow \gamma$. Since $v$ was an arbitrary point in $\mathcal{V}$, this shows that $\mathcal{F}_{\rightarrow}$ determines the orientation $\rightarrow$ on $\mathcal{F}$ uniquely, and our claim is proved. Equivalently, the map $\varphi $ is injective.

%\smallskip
%We are only left to prove that $\varphi $ is surjective. Let $\mathcal{S}$ be a Fano plane orthogonal to $\mathcal{F}$. By Lemma \ref{l2}, for a given $v$ in $\mathcal{V}$ there exist a unique triple $\alpha\beta\gamma$ in $\mathcal{F}$ and a unique triple $xyz$ in $\mathcal{S}$ that partition $\mathcal{V}$ together with $v.$ Let us set $x \rightarrow v$, $y \rightarrow v$, $z \rightarrow v$, and $v\rightarrow \alpha$, $v\rightarrow \beta$, $v\rightarrow \gamma$. We need to show that, as $v$ ranges over $\mathcal{V},$ this determines an orientation on $\mathcal{F}$ unambiguously, thus ensuring that the definition of $\rightarrow$ is well posed.

\smallskip
Since any triple in $\mathcal{F}$ different from $ijk$ intersects $ijk$ in precisely one point, we may assume that, for $v=i,$ there exist a triple $kmn$ in $\mathcal{F}$ and a triple $juw$ in $\mathcal{S}$ that partition $\mathcal{V}$ together with $i$, thereby inducing the orientation $j \rightarrow i \rightarrow k \rightarrow j$. Also, for $v=j$, there exist a triple $T_\mathcal{F}$ in $\mathcal{F}$ and a triple $T_\mathcal{S}$ in $\mathcal{S}$ that partition $\mathcal{V}$ together with $j$. Now $T_\mathcal{S}$ and $juw,$ being distinct triples in $\mathcal{S}$, have precisely one element in common (necessarily different from $j$), hence we may assume that $T_\mathcal{S}$ contains $u$ but not $w,$ that is, $w \in T_\mathcal{F}$ and $u \not \in T_\mathcal{F}.$ This implies, in turn, that $k \not \in T_\mathcal{F},$  since $kuw$ is the unique triple in $\mathcal{F}$ containing $k$ and $w.$ Hence $T_\mathcal{F} \cap \{i,j,k\} = \{i\}$ and $k \in T_\mathcal{S}$, say $T_\mathcal{F} = imw$ and $T_\mathcal{S}= kun,$ thus inducing again the orientation $j \rightarrow i \rightarrow k \rightarrow j$. Finally, in addition to $kuw$ and $imw,$ $\mathcal{F}$ must also contain $jnw,$ whereas $\mathcal{S}$, in addition to $juw$ and $kun,$ must also contain $ium,$ which is disjoint from $jnw$ and $\{k\}$. Therefore the choice $v=k$ produces again the orientation $j \rightarrow i \rightarrow k \rightarrow j$.
It follows that $\rightarrow$ is a well-defined orientation on $\mathcal{F}$, as claimed.
\endproof

Propositions \ref{p1} and \ref{p2} now immediately imply the following results.

\begin{cor}\label{eightorientations}
A given Fano plane admits exactly eight distinct orientations.
\end{cor}

\begin{cor}\label{orientationchoice}
Let $\rightarrow_1$ and $\rightarrow_2$ be two orientations on a Fano plane $\mathcal{F}$. Then there exists an automorphism $\sigma$ of $\mathcal{F}$ such that
$\sigma(\rightarrow_1) = \,\rightarrow_2$. As a consequence, the group of automorphisms of an oriented Fano plane does not depend on
the choice of the orientation.
\end{cor}

\proof
Let $ \rightarrow _1$ and $ \rightarrow _2$ be two orientations on $\mathcal{F}$, corresponding to the two pairs of orthogonal Fano planes $(\mathcal{F},\mathcal{F}_{\rightarrow_1})$ and $(\mathcal{F},\mathcal{F}_{\rightarrow_2})$, respectively, by Proposition \ref{p2}.
Now, by Proposition \ref{p1}, there exists an automorphism $\sigma$ of $\mathcal{F}$ that is an isomorphism from $\mathcal{F}_{\rightarrow_1}$ onto $\mathcal{F}_{\rightarrow_2}$. In particular, $\sigma(\mathcal{F}_{\rightarrow_1})=\mathcal{F}_{\rightarrow_2}$.
On the other hand, it is immediate by Proposition \ref{p2} that $\sigma(\mathcal{F}_{\rightarrow_1})=\mathcal{F}_{\sigma(\rightarrow_1)}$, hence $\sigma(\rightarrow_1) = \,\rightarrow_2$.

\smallskip
Therefore, if $\mu$ is an automorphism of $\mathcal{F},$ then $\mu$ is an automorphism of $(\mathcal{F}, \rightarrow_2)$ if and only if $\sigma^{-1}\mu\sigma$ is an automorphism of $(\mathcal{F}, \rightarrow _1)$. It follows that the group of automorphisms of $(\mathcal{F}, \rightarrow_2)$ is (isomorphic to) the group of automorphisms of $(\mathcal{F}, \rightarrow_1)$.
\endproof

\begin{cor}\label{orientedorthogonal} Let $\mathcal{F}$ be a Fano plane, let $\rightarrow$ be an orientation on $\mathcal{F}$, and let $\mathcal{F}_{\rightarrow}$ be the associated orthogonal Fano plane. If $\sigma$ is an automorphism of $\mathcal{F}$, then $\sigma$ is an automorphism of $(\mathcal{F}, \rightarrow)$ if and only if $\sigma$ is an automorphism of $\mathcal{F}_{\rightarrow}$. As a consequence, the automorphism group of the oriented Fano plane $(\mathcal{F}, \rightarrow)$ is isomorphic to the group of common automorphisms of the two orthogonal Fano planes $\mathcal{F}$ and $\mathcal{F}_{\rightarrow}$.
\end{cor}

\proof Let $\sigma$ be an automorphism of $\mathcal{F}$. Since $\mathcal{F}_{\rightarrow}$ is the orthogonal Fano plane associated with $\rightarrow$, $\sigma(\mathcal{F}_{\rightarrow})$ is the orthogonal Fano plane associated with $\sigma(\rightarrow)$. Hence $\sigma(\rightarrow) = \,\rightarrow$ if and only if $\sigma(\mathcal{F}_{\rightarrow}) = \mathcal{F}_{\rightarrow}$. Equivalently, by Definition \ref{automoriented}, $\sigma$ is an automorphism of $(\mathcal{F}, \rightarrow)$ if and only if $\sigma$ is an automorphism of $\mathcal{F}_{\rightarrow}$.
\endproof

In view of Corollaries \ref{orientationchoice} and \ref{orientedorthogonal}, and of Theorem \ref{OrthogonalGroup}, we can now state the main result of this section.

\begin{thm}\label{OrientedGroup}
The group of automorphisms of an oriented Fano plane is isomorphic to the Frobenius group of order $21$.
\end{thm}

\begin{rem}\label{cyclicorientation}
\emph{(a) Let $(\mathcal{F}, \rightarrow)$ be an oriented Fano plane. In particular, by Definition \ref{oriented}, $\rightarrow$ induces a cyclic orientation on each of the seven triples in $\mathcal{F}$. Moreover, the first part of the proof of Proposition \ref{p2} shows that $\rightarrow$ induces a cyclic orientation also on each of the seven triples $T(a), T(b),\dots, T(g)$ in $\mathcal{F}_{\rightarrow}.$ In fact, the orthogonal Fano plane $\mathcal{F}_{\rightarrow}$ can be characterized as the unique Fano plane $\mathcal{S}$, orthogonal to $\mathcal{F}$, with the property that the orientation $\rightarrow$ on $\mathcal{F}$ induces a cyclic orientation on each of the seven triples in $\mathcal{S},$ as well.}

\smallskip
\emph{As a consequence, besides the definition of $\mathcal{F}_{\rightarrow}$ in Proposition \ref{p2}, the triples in $\mathcal{F}_{\rightarrow}$ can also be determined as follows. Given two distinct points $a, b$ in the point-set $\mathcal{V}$, if $abc$ is the triple through $a$ and $b$ in $\mathcal{F}$, say $a\rightarrow b \rightarrow c \rightarrow a,$ then the triple through $a$ and $b$ in $\mathcal{F}_{\rightarrow}$ is precisely the triple $abx$, where $x$ is the unique point in $\mathcal{V}$, different from $c$, such that $a\rightarrow b \rightarrow x \rightarrow a$. Indeed, if $x$ were not unique, then $a\rightarrow b \rightarrow y \rightarrow a$ for an additional point $y,$ and thus, by Definition \ref{oriented} and Proposition \ref{p2}, $cxy$ ($= T(a)$) would be a triple in both $\mathcal{F}$ and $\mathcal{F}_{\rightarrow}$, against the fact that the two Fano planes are orthogonal.}

\smallskip
\emph{For instance, if $(\bar{\mathcal{F}}_1, \rightarrow)$ is the oriented Fano plane in Figure \ref{orientedfano2} (see Example \ref{exoriented}), then it is clear from the picture that the cyclic orientation $0 \rightarrow 1 \rightarrow x \rightarrow 0$ occurs if and only if either $x=3$ or $x=5,$ which corresponds to the fact that $013$ is a triple in
$\bar{\mathcal{F}}_1$ and $015$ is a triple in the orthogonal Fano plane $(\bar{\mathcal{F}}_1)_{\rightarrow}.$ Also, one easily finds that $(\bar{\mathcal{F}}_1)_{\rightarrow}$ is precisely the Fano plane $\bar{\mathcal{F}}_2$ introduced in Example \ref{ex1}. With a different viewpoint, the property that the orientation $\rightarrow$ induces a cyclic orientation on the triples of both the orthogonal Fano planes $\bar{\mathcal{F}}_1$ and $\bar{\mathcal{F}}_2$ is an immediate consequence of the fact that, for any two distinct points $x,y$ in $\mathbb{Z}_7,$ $x \rightarrow y$ if and only if $y-x \in \{1,2,4\}$ (mod $7$), which, in turn, yields the cyclic orientation $x\rightarrow x+1\rightarrow x+3\rightarrow x$ and $x\rightarrow x+1\rightarrow x+5\rightarrow x$ on all the triples in $\bar{\mathcal{F}}_1$ and $\bar{\mathcal{F}}_2,$ respectively. A nice graphical visualization of this property will be proposed in the subsequent Remark \ref{cylicorientationtorus} in Section \ref{SecEmbeddings}.}

\smallskip
\emph{(b) The above characterization of the orthogonal Fano plane $\mathcal{F}_{\rightarrow}$ in remark (a), together with Definition \ref{oriented} and Proposition \ref{p2}, shows that $\mathcal{F}_{\rightarrow}$ is itself an oriented Fano plane with respect to the inverse orientation $\twoheadrightarrow,$ where $x \twoheadrightarrow y$ if and only if $y \rightarrow x.$ Also, by iterating the construction of the orthogonal Fano plane associated with the orientation, it is immediate that $(\mathcal{F}_{\rightarrow})_{\twoheadrightarrow} = \mathcal{F},$ by Proposition \ref{p2}.}

\smallskip
\emph{Moreover, we may give a simple alternative proof of Theorem \ref{OrientedGroup}. Indeed, by the characterization of $(\bar{\mathcal{F}}_1)_{\rightarrow}$ in remark (a), the block-set of $\bar{\mathcal{F}}_1$ and the block-set of $(\bar{\mathcal{F}}_1)_{\rightarrow}$ are both invariant under any automorphism of $(\bar{\mathcal{F}}_1, \rightarrow).$ Therefore the automorphism group $H$ of $(\bar{\mathcal{F}}_1, \rightarrow)$ is a subgroup of the group $G$ of common automorphisms of the orthogonal Fano planes $\bar{\mathcal{F}}_1$ and $(\bar{\mathcal{F}}_1)_{\rightarrow},$ which is isomorphic to $F_{21}$ by Theorem \ref{OrthogonalGroup}. On the other hand, $H$ contains the automorphisms $\lambda_2$ and $\tau_1$ (see Remark \ref{120degree}), which have order $3$ and $7,$ respectively, hence the order of $H$ is a multiple of $21.$ It finally follows that $H$ $(=G)$ is isomorphic to $F_{21}.$}

\smallskip
\emph{(c) Let $\mathcal{F}$ be a Fano plane with the same seven triples as in the proof of Proposition \ref{p2}. The construction in the proof shows that the orientation on $\mathcal{F}$ is uniquely determined by the orientation on the triples $abc, ade,$ and $cdg$. Since there are two possible cyclic orientations for each of these three triples, this gives an alternative proof that there exist exactly eight distinct orientations on $\mathcal{F}$ (Corollary \ref{eightorientations}).}

\smallskip
\emph{(d) We finally outline an alternative construction to characterize the orientations on a given Fano plane $\mathcal{F}=(\mathcal{V},\mathcal{B}).$ Let us say that a sequence $$\mathcal{C}=(x_1,\, x_2,\, x_3,\, x_4,\, x_5,\, x_6,\, x_7)$$ in $\mathcal{V}$ is a \emph{Fano circuit} in $\mathcal{F}$ if $\{x_1,\, x_2,\, x_3,\, x_4,\, x_5,\, x_6,\, x_7\} = \mathcal{V}$ and for each $T$ in $\mathcal{B}$ there exists a (necessarily unique) $2$-set $\{x_i,\, x_{i+1}\} \pmod 7$ such that $\{x_i,\, x_{i+1}\} \subseteq T.$ In other words, $\mathcal{C}$ is a Hamiltonian circuit in $K_7$, with the property that the edges $\{x_i,\, x_{i+1}\}$ corresponding to pairs of consecutive vertices ``cover'' all the seven triples in $\mathcal{B}.$}

\smallskip
\emph{Given a Fano circuit $\mathcal{C}=(x_1,\, x_2,\, x_3,\, x_4,\, x_5,\, x_6,\, x_7)$ in $\mathcal{F}=(\mathcal{V},\mathcal{B}),$ one can define a binary antisymmetric relation $\rightarrow$ on $\mathcal{V}$ by letting $x_i \rightarrow x_{i+1} \pmod 7$ for each $i$ and extending the relation cyclically on the block containing $x_i$ and $x_{i+1}.$ It is easy to prove that, up to replacing $\rightarrow$ by the inverse orientation $x_{i+1} \twoheadrightarrow x_i$ induced by the backward circuit $(x_7,\, x_6,\, \ldots ,\, x_1),$ $\rightarrow$ is an orientation on $\mathcal{F}$ in the sense of Definition \ref{oriented}. Moreover, for each $i,$ the $3$-set $\{x_i,\, x_{i+1},\, x_{i+2}\} \pmod 7$ is not a block in $\mathcal{F}$ nor in the orthogonal Fano plane $\mathcal{F}_{\rightarrow}.$}

\smallskip
\emph{Conversely, let $\rightarrow$ be an orientation on $\mathcal{F}=(\mathcal{V},\mathcal{B}).$ Then there exists a Fano circuit in $\mathcal{F}$ that induces $\rightarrow$ in the sense just described above. Indeed, let $x_1$ be a given point in $\mathcal{V},$ and let $\{a,b,c\}$ be the triple in $\mathcal{B}$ with the property that $x_1 \rightarrow a,$ $x_1 \rightarrow b,$ and $x_1 \rightarrow c.$ We choose a point in $\{a,b,c\}$ and we denote it by $x_2.$ Then $x_1 \rightarrow x_2$ and, in view of remark (a) above, there exists a unique point $x$ in $\mathcal{V}$ with the property that $x_2 \rightarrow x$ and the $3$-set $\{x_1,\, x_{2},\, x\}$ is not a block in $\mathcal{F}$ nor in the orthogonal Fano plane $\mathcal{F}_{\rightarrow}.$ We set $x_3=x.$ By iterating this algorithm, one constructs a sequence $\mathcal{C}=(x_1,\, x_2,\, x_3,\, x_4,\, x_5,\, x_6,\, x_7),$ which turns out to be a Fano circuit in $\mathcal{F}$ and which, by construction, induces the orientation $\rightarrow$. Note that $\mathcal{C}$ is uniquely determined by the choice of $x_2.$ Since there are three possible choices for $x_2,$ it follows that there exist precisely three Fano circuits in $\mathcal{F}$ that induce the orientation $\rightarrow$. On the other hand, it can be shown that, by identifying each Fano circuit with its backward circuit, there exist exactly $24$ Fano circuits in $\mathcal{F}.$ This finally gives yet another alternative proof that $\mathcal{F}$ admits exactly eight distinct orientations (Corollary \ref{eightorientations}).}
\end{rem}

\begin{rem}\label{remarkoctonions}
\emph{It is worth mentioning that the oriented Fano plane completely describes the algebraic structure of the octonions, which are an $8$-di\-men\-sional algebra over the real numbers, with basis $1,$ $e_1,$ $e_2,$ $e_3,$ $e_4,$ $e_5,$ $e_6,$ $e_7,$ and multiplication table given as in \cite[Table 1]{Baez} (which is known as the Cartan-Schouten-Coxeter octonion multiplication table). The octonions constitute, up to isomorphism, the unique non-associative and non-commutative normed division algebra. The basis element $1$ is also the multiplicative identity, $e_1,\ldots,e_7$ are square roots of $-1,$ and the remaining products are cyclically generated (mod $7$) by the basic products $e_1e_2=e_4$ and $e_2e_1=-e_4.$}

\smallskip
\emph{If we replace the symbols $1,2,\ldots,6$ in Figure \ref{orientedfano2} by $e_1, e_2,\ldots,e_6,$ respectively, and the symbol $0$ by $e_7,$ then the orientation of the blocks of the Fano plane completely determines the product of any two distinct basis elements different from the multiplicative identity $1$: if $\{e_i,e_j,e_k\}$ is a block with orientation $e_i \rightarrow e_j \rightarrow e_k,$ then $e_ie_j=e_k$ and $e_je_i=-e_k$ \cite[p. 152]{Baez} (an earlier version of the same construction can be found in \cite[Figure 2]{Manogue}).}

\smallskip
\emph{It is then immediate that Theorem \ref{OrientedGroup} above on the automorphisms of the oriented Fano plane can be seen as a reformulation of the following result on the automorphisms of the octonions. This result was probably known before the undergraduate thesis \cite{Killgore} that we are quoting here, although no earlier proof seems to have appeared in print.}
\end{rem}

\begin{cor}\label{Frobeniusoctonions} \emph{\cite[\S 4.4]{Killgore}} The group of the automorphisms of the octonions that permute the imaginary basis units $e_1, e_2,e_3,e_4,e_5,e_6,e_7$ is isomorphic to the Frobenius group of order $21.$
\end{cor}

\section{Combinatorial Embeddings of Graphs}\label{SecEmbeddings}
In this section we prove that any triangular embedding of the complete graph $K_7$ into a surface is isomorphic to the classical biembedding into the $2$-torus and hence is face $2$-colorable. Moreover, we show that the group of the embedding automorphisms that preserve the color classes is isomorphic to $F_{21}$.

\smallskip
Let us recall that, if $\psi$ is a biembedding, then the group of the embedding automorphisms that fix the color classes is denoted by $\operatorname{Aut}(\psi,Col)$ (see the terminology introduced in Definition \ref{defembiso}).

\begin{thm}\label{classicalemb}
If $\varphi: K_7 \rightarrow \mathbb{T}^2$ is the classical biembedding of $K_7$ into the $2$-torus, then $\operatorname{Aut}(\varphi,Col)$ is isomorphic to the Frobenius group of order $21$.
\end{thm}

\proof
Let $\mathbb{Z}_7 = \{0,1,2,3,4,5,6\}$ be the vertex-set of $K_7,$ and let $\varphi$ be the embedding of $K_7$ into the $2$-torus presented in Figure \ref{fig1}. By Definition \ref{defembiso}, an embedding automorphism (of $\varphi$) is a permutation $\sigma$ of $\mathbb{Z}_7$ that permutes the faces of $K_7$ induced by $\varphi,$ which are, in turn, the blocks of the Fano planes $\bar{\mathcal{F}}_1=(\mathcal{V},\mathcal{\bar{B}}_1)$ and $\bar{\mathcal{F}}_2=(\mathcal{V},\mathcal{\bar{B}}_2)$ introduced in Example \ref{ex1}. Since the two color classes of $\varphi$ are precisely $\mathcal{\bar{B}}_1$ and $\mathcal{\bar{B}}_2,$ we conclude that $\sigma$ is in $\operatorname{Aut}(\varphi,Col)$ if and only if $\sigma$ is a common automorphisms of the orthogonal Fano planes $\bar{\mathcal{F}}_1$ and $\bar{\mathcal{F}}_2.$ The thesis is now an immediate consequence of Theorem \ref{OrthogonalGroup}.
\endproof

\smallskip
More in general, in the remaining part of this section, we investigate the triangular embeddings of $K_7$ and their automorphisms. First of all, following \cite{GG,GT, JS}, we provide an equivalent, but purely combinatorial, definition of a graph embedding into a surface, assuming that the surface is orientable.
Given a graph $\Gamma = (\mathcal{V},\mathcal{E}),$ we denote by $\mathcal{D}(\Gamma)$ the set of all the ordered pairs $(x,y)$ and $(y,x)$, as $\{x,y\}$ ranges over the edge-set $\mathcal{E},$ and we say that the elements of $\mathcal{D}(\Gamma)$ are the \emph{oriented edges} of $\Gamma$ (note that the cardinality of $\mathcal{D}(\Gamma)$ is twice the cardinality of $\mathcal{E},$ and that we do not assume any prescribed orientation on the edges of $\Gamma$). Also, given a vertex $x$ in $\mathcal{V}$, we denote by $\mathcal{N}(\Gamma,x)$ the set of all the vertices adjacent to $x$ in $\Gamma$.

\begin{defn}\label{DefEmbeddings}
\emph{Let $\Gamma = (\mathcal{V},\mathcal{E})$ be a connected graph. A \emph{combinatorial embedding} of $\Gamma$ is a pair $\Pi=(\Gamma,\rho)$ where $\rho: D(\Gamma)\rightarrow D(\Gamma)$ is a map satisfying the following property: for any $x$ in $\mathcal{V},$ there exists a permutation $\rho_x$ of $\mathcal{N}(\Gamma,x)$, which is a cycle of order $|\mathcal{N}(\Gamma,x)|$ and is such that, for any $y\in \mathcal{N}(\Gamma,x)$, $\rho(x,y)=(x,\rho_x(y)).$ If this is the case, then the map $\rho$ is said to be a \emph{rotation} of $\Gamma$.}
\end{defn}

%\smallskip
As reported in \cite{GG}, a cellular embedding of a (connected) graph $\Gamma$ into an orientable surface is equivalent to a combinatorial embedding $\Pi=(\Gamma,\rho)$ (see also \cite[Theorem 3.1]{A} and \cite{DP}). This yields, in particular, a purely combinatorial description of the faces of $\Gamma$ induced by the embedding. Indeed, in the case of a combinatorial embedding $(\Gamma,\rho)$, an edge $\{x,y\}$ belongs to exactly two faces $T_1$ and $T_2,$ which can be determined via the so-called \emph{face-tracing algorithm} as follows:

\begin{equation}\label{facetracing}
\begin{array}{ccl}
  T_1 & = & (x,\rho_x(y),\rho_{\rho_x(y)}(x),\dots, y) \\
  & & \\
  T_2 & = & (y,\rho_y(x),\rho_{\rho_y(x)}(y),\dots,x).
\end{array}
\end{equation}

%\begin{equation}\label{facetracing}T_1=(x,\rho_x(y),\rho_{\rho_x(y)}(x),\dots, y)
%\end{equation}
%\begin{equation*}T_2=(y,\rho_y(x),\rho_{\rho_y(x)}(y),\dots,x).\end{equation*}

\smallskip
We also recall that the notions of embedding isomorphism and embedding automorphism can be defined purely combinatorially as well, as follows (see Korzhik and Voss \cite{Korzhik}, page 61).

\begin{defn}\label{DefEmbeddingsIs}
\emph{Let $\Gamma$ and $\Gamma'$ be two connected graphs, and let $\Pi= (\Gamma,\rho)$ and $\Pi'= (\Gamma',\rho')$ be two combinatorial embeddings of $\Gamma$ and $\Gamma'$, respectively. We say that $\Pi$ is \emph{isomorphic} to $\Pi'$ if there exists a graph isomorphism $\sigma: \Gamma\rightarrow \Gamma'$ such that either
\begin{equation}\label{eq11}
\sigma\circ \rho(x,y)=\rho'\circ \sigma(x,y)\;\;\;\;\;\;\;\;\;\; \mbox{for any $(x,y)\in D(\Gamma)$}
\end{equation}
or
\begin{equation}\label{eq12}
\sigma\circ \rho(x,y)=(\rho')^{-1}\circ \sigma(x,y)\;\;\;\;\; \mbox{for any $(x,y)\in D(\Gamma)$}.
\end{equation}}

\emph{We also say that $\sigma$ is an \emph{embedding isomorphism} between $\Pi$ and $\Pi'$.
Moreover, if equation (\ref{eq11}) holds, then $\sigma$ is said to be an \emph{orientation-preserving isomorphism}, whereas,
if (\ref{eq12}) holds, then $\sigma$ is said to be an \emph{orientation-reversing isomorphism}. Also, if $\Gamma'=\Gamma$ and $\rho'=\rho$ (that is, $\Pi'=\Pi$), then we say that $\sigma$ is an \emph{embedding automorphism} of $\Pi.$ We denote by $\operatorname{Aut}(\Pi)$ the group of all the embedding automorphisms of $\Pi.$ Finally, if $\Pi$ admits a face $2$-coloring, then we denote by $\operatorname{Aut}(\Pi,Col)$ the group of the embedding automorphisms of $\Pi$ that fix the color classes.}
\end{defn}

\smallskip
Let us first consider the classical toroidal embedding of $K_7$, seen as a combinatorial embedding. As explained in \cite{MT}, the rotation $\rho$ associated with the biembedding of $K_7$ in Figure \ref{fig1}, with vertex-set $\mathbb{Z}_7 = \{0,1,2,3,4,5,6\},$ can be obtained through the two following steps.

\begin{itemize}
\item[1)] Given $x\in \mathbb{Z}_7$, the map $\rho_x$ is defined by looking at the neighbors of $x$ in a clockwise order.
For example, for $x=3,$ $\rho_3$ is the cyclic permutation $(0\:2\:5\:6\:4\:1)$ of $\mathbb{Z}_7$.

\item[2)] By Definition \ref{DefEmbeddings}, $\rho(x,y)= (x,\rho_x(y))$ for any $(x,y)\in D(\Gamma)$.
\end{itemize}

\smallskip
These two steps are enough to determine the map $\rho$. Alternatively, however, for a given $x\in \mathbb{Z}_7,$ one can define the map $\rho$ by clockwise rotation just on the oriented edges starting from $x$, and then extend the map $\rho$ on the remaining oriented edges as follows.
\begin{itemize}
\item[3)] Given a vertex $x$ in $\mathbb{Z}_7,$ every oriented edge of $K_7$ can be written in the form $(x+z,y+z),$ where the sum is performed modulo $7$. Then we set
    \begin{equation*}(\star) \ \ \ \ \ \ \ \rho(x+z,y+z)=(x+z,\rho_x(y)+z).\end{equation*}
\end{itemize}

By doing so, one can easily check that the expression $(\star)$ for $\rho$ is independent of the initial vertex $x$. Also, if $\tau_z$ is the permutation of $\mathbb{Z}_7$ defined by $\tau_z(x)=x+z \pmod 7$, then property $(\star)$ can be written as
$$\rho\circ \tau_z(x,y) = \tau_z\circ \rho(x,y).$$

\smallskip
By Definition \ref{DefEmbeddingsIs}, this means that, for any $z\in \mathbb{Z}_7$, the map $\tau_z$ is an orientation-preserving automorphism of the combinatorial embedding $\bar{\Pi} = (K_7,\rho)$. In particular, following the notation introduced in the same definition, $\tau_z \in \operatorname{Aut}(K_7,\rho)$ for any $z\in \mathbb{Z}_7$. In addition, the embedding $(K_7,\rho)$ is $\mathbb{Z}_7$-regular, that is, it admits $\mathbb{Z}_7$ as a strictly transitive group of automorphisms on the vertices of $K_7.$ Note that, on the torus, the permutations $\tau_z$ can be visualized as one-step translations along the ``lines'' in Figure \ref{fig1} (see the subsequent Remark \ref{rototranslations}). Finally recall that, by Definition \ref{DefEmbeddingsIs}, $\operatorname{Aut}(\bar{\Pi},Col)$ denotes the group of the embedding automorphisms of $\bar{\Pi} = (K_7,\rho)$ that fix the color classes.

\smallskip
We can now state an equivalent version of Theorem \ref{classicalemb}.

\begin{thm}\label{mainthmemb}
If $\bar{\Pi} = (K_7,\rho)$ is the combinatorial embedding of Figure \ref{fig1}, then $\operatorname{Aut}(\bar{\Pi},Col)\cong F_{21}$.
\end{thm}

\proof
Since the combinatorial embedding $\bar{\Pi} = (K_7,\rho)$ is equivalent to the embedding $\varphi: K_7 \rightarrow \mathbb{T}^2$ of Figure \ref{fig1}, the thesis immediately follows from Theorem \ref{classicalemb}.

\smallskip
One can also give a direct and independent proof as follows. As usual, let $\mathbb{Z}_7 = \{0,1,2,3,4,5,6\}$ be the vertex-set of $K_7.$ For any non-zero $c$ in $\mathbb{Z}_7,$ let $\lambda_c$ be the permutation of the vertices defined by $\lambda_c(x)=cx \pmod{7}.$
One can check directly, by inspection, that the maps $\lambda_c,$ $c \neq 0,$ are orientation-preserving automorphisms of $\bar{\Pi}.$
This can be proved, alternatively, as follows. By looking at Figure \ref{fig1}, one readily finds that $\rho_0$ is the cyclic permutation
$(1\:5\:4\:6\:2\:3).$ Equivalently, $\rho_0 = \lambda_{5}$. It now follows from property $(\star)$ above that, for any $x$ and any $y \neq x,$ $$\rho(x,y) = \rho(0+x, (y-x)+x) = (0+x, \rho_0(y-x)+x) = (x,5y-4x),$$ hence $\rho(cx,cy)=(cx,5cy-4cx)=c\rho(x,y),$ that is,
$\rho \lambda_{c} = \lambda_{c} \rho,$ for all $c.$ Equivalently, by Definition \ref{DefEmbeddingsIs}, $\lambda_{c}$ is an orientation-preserving automorphism of $\bar{\Pi}$ for all $c.$

\smallskip
Now the map $\lambda_2$ is an element of $\operatorname{Aut}(\bar{\Pi},Col)$ (see Example \ref{lambdatwo}). Moreover, we noted earlier that $\tau_1$ is an automorphism of $\bar{\Pi}.$ In fact, by definition of the coloring of $\bar{\Pi}$, $\tau_1$ fixes the color classes and hence is an element of $\operatorname{Aut}(\bar{\Pi},Col)$. Since $\lambda_2$ has order three and $\tau_1$ has order seven, the order of $\operatorname{Aut}(\bar{\Pi},Col)$ must be a multiple of $21$.

\smallskip
On the other hand, the color classes define a pair of orthogonal Fano planes (see Example \ref{ex1}). Therefore it follows from Theorem \ref{OrthogonalGroup} that $\operatorname{Aut}(\bar{\Pi},Col)$ is isomorphic to a subgroup of $F_{21}$. Since its size is a multiple of $21$, we conclude that $\operatorname{Aut}(\bar{\Pi},Col)\cong F_{21},$ as claimed.
\endproof

\begin{rem}\label{rototranslations} \emph{The proof of Theorem \ref{mainthmemb} shows that the group $\operatorname{Aut}(\bar{\Pi},Col)\cong F_{21}$ is generated by the permutations $\lambda_2(x)=2x$ and $\tau_1(x)=x+1 \pmod 7$ of $\mathbb{Z}_7 = \{0,1,2,3,4,5,6\}$ (note the analogy with Remark \ref{remlambdatau}). This gives, by means of Figure \ref{fig1}, a very simple and elegant way to visualize the $21$ permutations $x \mapsto ax+b$ ($a \in \{1,2,4\},$ $b \in \mathbb{Z}_{7}$) in $F_{21}$ as a group of transformations of the torus.}

\smallskip
\emph{Indeed, every permutation of the form $x \mapsto 2x+b$ (resp., $x \mapsto 4x+b$) can be seen as an order-$3$ counterclockwise (resp., clockwise) rotation around the point $-b$ (resp., $2b$), whereas every permutation of the form $x \mapsto x+b,$ $b \neq 0,$ can be seen as a one-step translation along some ``line'' in Figure \ref{fig1}. For instance, the permutation $x \mapsto 2x+4$ (resp., $x \mapsto 4x+5$) represents an order-$3$ counterclockwise (resp., clockwise) rotation around the point $3,$ whereas the permutation $x \mapsto x+3$ (resp., $x \mapsto x-3$) can be visualized as a one-step upward (resp., downward) translation along the ``lines'' in Figure \ref{fig1} parallel to the line containing $0,3,$ and $6.$}
\end{rem}

\begin{rem}\label{cylicorientationtorus} \emph{We noted in Remark \ref{cyclicorientation}(a) that the orientation $\rightarrow$ defined in Example \ref{exoriented} (also represented in Figure \ref{orientedfano2}) induces a cyclic orientation on the triples of both the orthogonal Fano planes $\bar{\mathcal{F}}_1$ and $\bar{\mathcal{F}}_2$ introduced in Example \ref{ex1}. This property can be easily visualized by means of the representation of the blocks in $\bar{\mathcal{F}}_1$ and $\bar{\mathcal{F}}_2$ as the grey and white triangular faces in Figure \ref{fig1}, respectively. Indeed, for any pair of distinct points $x,y$ in $\mathbb{Z}_7 = \{0,1,2,3,4,5,6\},$ let us reproduce on the edge $\{x,y\}$  in Figure \ref{fig1} the same arrow sign that appears in Figure \ref{orientedfano2} (equivalently, the arrow mark points from $x$ to $y$ if and only if $y-x \in \{1,2,4\}$ (mod $7$)). Then, as a result, every grey (resp. white) face ends up displaying a cyclic counterclockwise (resp. clockwise) orientation. For instance, there is a counterclockwise orientation $0 \rightarrow 1 \rightarrow 3 \rightarrow 0$ on the grey face $\{0,1,3\}$ in the bottom-right corner of the picture, and a clockwise orientation $0 \rightarrow 4 \rightarrow 6 \rightarrow 0$ on the white face $\{0,4,6\}$ in the upper-left corner. Finally, note that the orientation $\rightarrow$ is induced by the Fano circuit $\mathcal{C}=(0,\, 1,\, 2,\, 3,\, 4,\, 5,\, 6),$ in the sense of Remark \ref{cyclicorientation}(d).}
\end{rem}

\smallskip
In the final part of this section, we generalize Theorems \ref{classicalemb} and \ref{mainthmemb} to an arbitrary triangular embedding of $K_7$ into a surface. We prove that, no matter how we take a triangular embedding of $K_7,$ the embedding is face $2$-colorable and isomorphic to the classical toroidal (combinatorial) embedding $\bar{\Pi} = (K_7,\rho)$. We first consider the case where the surface is orientable. As we recalled earlier in this section, this is equivalent to considering a triangular combinatorial embedding of $K_7$.

\begin{lem}\label{3emb}
Let $\Pi'=(K_7,\rho')$ be a triangular combinatorial embedding of $K_7$. If $\rho'_x(y)=z$, then $\rho'_z(x)=y$ and $\rho'_y(z)=x$.
\end{lem}

\begin{proof}
Suppose that $\rho'_x(y)=z$. Because of the face-tracing algorithm (see \eqref{facetracing}), the face containing $x,y,$ and $z$ is
$$(x,\rho'_x(y)=z,\rho'_{z}(x),\dots,y).$$

\smallskip
Since the embedding is triangular, it follows that $\rho'_z(x)=y$. Similarly, starting from $z,$ the same face can be written as
$$(z, \rho'_{z}(x)=y,\rho'_{y}(z),\dots,x).$$

\smallskip
Since, again, the embedding is triangular, we conclude that $\rho'_y(z)=x$.
\end{proof}

\begin{prop}\label{orientable}
Let $\Pi'=(K_7,\rho')$ be a triangular combinatorial embedding of $K_7$. Then $\Pi'$ is isomorphic to the classical toroidal embedding $\bar{\Pi} = (K_7,\rho)$.
\end{prop}

\begin{proof}
As usual, we identify the vertex-set of $K_7$ with $\mathbb{Z}_7 = \{0,1,2,3,4,5,6\}.$ In view of Definition \ref{DefEmbeddingsIs}, up to replacing $\rho'$ by $\sigma \rho' \sigma^{-1},$ where $\sigma$ is a suitable permutation of the vertices, we may assume that $$\rho'_0=(1\;5\;4\;6\;2\;3),$$ as in $\bar{\Pi}$. In particular, $\rho'_0(5)=4$ and $\rho'_0(1)=5.$ Therefore, since the embedding is triangular, it follows from Lemma \ref{3emb} that $\rho'_5$ is of the form $(4\;0\;1\;x\;y\;z),$ where $\{x,y,z\}=\{2,3,6\}$.
Suppose, by contradiction, that $\rho'_5(1)=3$ (resp., $\rho'_5(6)=4$). Together with $\rho'_0(1)=5$ and $\rho'_0(3)=1$ (resp., $\rho'_0(5)=4$ and $\rho'_0(4)=6$), this would imply, by Lemma \ref{3emb}, that the permutation $\rho'_1$ (resp., $\rho'_4$) contains the $3$-cycle $(3\;5\;0)$ (resp., $(5\;6\;0)$), contradicting the fact that $\rho'_1$ (resp., $\rho'_4$) has order six by Definition \ref{DefEmbeddings}. Hence $x \neq 3$ and $z \neq 6$. Now  $\rho'_0(6)=2$ implies, by Lemma \ref{3emb}, that $\rho'_2(0)=6$; hence $\rho'_2(5) \neq 6,$ which implies, in turn, that $\rho'_5(6)\neq 2.$ It finally follows that $\rho'_5$ is either $(4\;0\;1\;6\;3\;2)$ or $(4\;0\;1\;2\;6\;3)$.

\smallskip
In the former case, the conditions $\rho'_5(1)=6$, $\rho'_0(1)=5,$ and $\rho'_0(3)=1$ imply, by Lemma \ref{3emb}, that $\rho'_1(6)=5$, $\rho'_1(5)=0$, and $\rho'_1(0)=3$, respectively, hence $\rho'_1$ is of the form $(6\;5\;0\;3\;x\;y)$, where $\{x,y\} = \{2,4\}.$ Now $\rho'_5(3)=2 \Rightarrow \rho'_3(2)=5 \Rightarrow \rho'_3(2) \neq 1 \Rightarrow x=\rho'_1(3) \neq 2,$ thus $\rho'_1 = (6\;5\;0\;3\;4\;2).$ By reasoning in the same way, one finds that $\rho'_2 = (0\;6\;1\;4\;5\;3),$ $\rho'_3 = (4\;1\;0\;2\;5\;6),$ $\rho'_4 = (6\;0\;5\;2\;1\;3),$ and $\rho'_6 = (3\;5\;1\;2\;0\;4).$ Hence $\rho'=\rho,$ and $\Pi'$ is isomorphic to $\bar{\Pi}$.

\smallskip
Let us finally assume that $\rho'_5=(4\;0\;1\;2\;6\;3)$. By arguing as in the previous case, one finds that
$$\rho'_1=(2\;5\;0\;3\;6\;4)$$
$$\rho'_2=(3\;0\;6\;5\;1\;4)$$
$$\rho'_3=(1\;0\;2\;4\;5\;6)$$
$$\rho'_4=(2\;1\;6\;0\;5\;3)$$
$$\rho'_6 \:=(1\;3\;5\;2\;0\;4).$$

\smallskip
Now the permutation $\sigma = (2\;4)(3\;5)$ of the vertex-set $\{0,1,2,3,4,5,6\}$ is such that
$$\sigma\circ \rho'(x,y)=\rho^{-1}\circ \sigma(x,y)$$ for all oriented edges $(x,y)$ in $K_7,$ as it can be easily checked by inspection.
Hence, by Definition \ref{DefEmbeddingsIs}, $\Pi'=(K_7,\rho')$ is isomorphic to $\bar{\Pi}=(K_7,\rho),$ as claimed.
\end{proof}

\smallskip
We finally consider the case of an arbitrary triangular embedding of $K_7$ into a surface, without assuming that the surface be orientable.

\smallskip
\begin{cor}
Let $\psi$ be a triangular embedding of $K_7$ into a (non-necessarily orientable) surface $\Sigma$. Then the embedding admits a unique face $2$-coloring and $$\operatorname{Aut}(\psi,Col) \cong F_{21}.$$
\end{cor}

\begin{proof}
As in the orientable case, each edge of $K_7$ belongs to exactly two faces. Indeed, the two faces can be determined by means of the face-tracing algorithm induced by a suitable combinatorial embedding, which can be associated to $\psi$ also in the non-orientable case (by generalizing \cite[Definition 5.2]{CostaMellaPasotti}). Since the faces of $\psi$ are triangular, the total number of faces is $14$. Then, by Euler's formula, the Euler characteristic of $\Sigma$ is
$$\chi=F+V-E=14+7-21=0.$$

Therefore $\Sigma$ is either the torus or the Klein bottle. On the other hand, it is well known that $K_7$ does not admit embeddings into the Klein bottle (see, e.g., Theorem 4.11, p. 69, in \cite{Map}). Therefore $\Sigma$ is orientable and hence, because of Proposition \ref{orientable}, $\psi$ is isomorphic to the classical toroidal (combinatorial) embedding $\bar{\Pi} = (K_7,\rho)$. The thesis now follows immediately from Theorem \ref{mainthmemb} and from the uniqueness of the face $2$-coloring of $\bar{\Pi}$.
\end{proof}

\section*{Acknowledgements}
The authors were partially supported by INdAM--GNSAGA.
The second author was supported also by Universit\`a di Palermo (FFR2024 Pavone).


\begin{thebibliography}{50}

\bibitem{Anderson3} I.~Anderson, {\it Combinatorial designs and tournaments}, Oxford University Press, Oxford, 1997.
\bibitem{A} D. S. Archdeacon,
\textit{Heffter arrays and biembedding graphs on surfaces}, Electron. J. Combin. \textbf{22} (2015) \#P1.74.

\bibitem{Baez} J.~C.~Baez, \textit{The octonions}, Bull. Amer. Math. Soc.~\textbf{39} (2) (2002), 145--205.

\bibitem{BJL} T.~Beth, D.~Jungnickel, H.~Lenz, {\it Design theory}, 2nd ed., Cambridge University Press, Cambridge, 1999.

\bibitem{CRC} C.~J.~Colbourn, J.~H.~Dinitz (eds.), {\it The CRC Handbook of Combinatorial Designs}, 2nd ed., Chapman and Hall/CRC Press, Boca Raton, 2007.

\bibitem{ColbournRosa} C.~J.~Colbourn, A.~Rosa, {\it Triple systems}, Oxford University Press, Oxford, 1999.

\bibitem{Cole} F.~N.~Cole, \textit{Kirkman parades}, Bull. Amer. Math. Soc. \textbf{28} (1922), 435--437.

\bibitem{CostaPasotti2} S. Costa, A. Pasotti, \textit{On the number of non-isomorphic (simple) $k$-gonal biembeddings of complete multipartite graphs}, accepted in Ars. Math. Contemp., preprint available at https://arxiv.org/abs/2111.08323.

\bibitem{CostaMellaPasotti} S. Costa, L. Mella, A. Pasotti, \textit{Weak Heffter Arrays and biembedding graphs
on non-orientable surfaces}, Electron. J. Combin. \textbf{31} (1) (2024) \#P1.8.

\bibitem{Denniston} R.~H.~F.~Denniston, \textit{Sylvester's problem of the 15 schoolgirls}, Discrete Math.~\textbf{9} (3) (1974), 229--233.

\bibitem{Gibbons} P.~B.~Gibbons, \textit{A Census of Orthogonal Steiner Triple Systems of Order $15$}, Annals of Discrete Mathematics~\textbf{26} (1985), 165--182.

\bibitem{GG} M. J. Grannell, T. S. Griggs, \textit{Designs and topology}, Surveys in Combinatorics 2007, London Mathematical
Society Lecture Note Series, 346 (A. Hilton and J. Talbot, eds.), Cambridge University Press, Cambridge (2007), 121--174.

\bibitem{GGS} M. J. Grannell, T. S. Griggs, J. Siran, \textit{Surface Embeddings of Steiner Triple Systems}, J.
Combin. Des. ~\textbf{6} (1998), 325-336.

\bibitem{GT} J. L. Gross, T. W. Tucker, \textit{Topological Graph Theory}, John Wiley, New York, 1987.

\bibitem{Killgore} P.~L.~Killgore, \textit{The geometry of the
octonionic multiplication table}, Thesis, Oregon State University, Corvallis, 2015.

\bibitem{Lady} T.~P.~Kirkman, Query VI, {\it Lady's and Gentleman's Diary} (1850).

\bibitem{Korzhik}V. P. Korzhik, H. J. Voss, \textit{On the Number of Nonisomorphic Orientable Regular Embeddings of Complete Graphs}, J. Combin. Theory Ser. B \textbf{81} (2001), 58--76.

\bibitem{Manogue} C. A. Manogue, J. Schray,\textit{ Finite Lorentz transformations, automorphisms, and division algebras}, J. Math. Phys. \textbf{34} (1993), 3746--3767.

\bibitem{MPR} R. A. Mathon, K. T. Phelps, A. Rosa, \textit{Small Steiner triple systems and their properties}, Ars Combin. \textbf{15} (1983), 3--110.

\bibitem{Moh} B. Mohar, \textit{Combinatorial local planarity and the width of graph embeddings}, Canad. J. Math. \textbf{44} (1992), 1272--1288.

\bibitem{MT} B. Mohar, C. Thomassen, \textit{Graphs on surfaces}, Johns Hopkins University Press, Baltimore, 2001.

\bibitem{Mullin} R. C. Mullin, E. Nemeth,\textit{ On the non-existence of orthogonal Steiner systems of order 9}, Can. Math. Bull. \textbf{13} (1970), 131--134.

\bibitem{Shaughnessy} C. D. O'Shaughnessy, \textit{A Room design of order $14$}, Can. Math. Bull. \textbf{11} (1968), 191--194.

\bibitem{DP} A. Pasotti and J. H. Dinitz, \textit{A survey of Heffter arrays}, in: Stinson 66 – New Advances in Designs, Codes and Cryptography,
C. J. Colbourn and J. H. Dinitz (eds.), Fields Institute Communications \textbf{86} (2024), 353--392.

\bibitem{Seven} M. Pavone, \textit{On the seven non-isomorphic solutions of the fifteen schoolgirl problem}, Discrete Math.~\textbf{346} (6) (2023) 113316.

\bibitem{Map} G. Ringel,\textit{ Map Color Theorem}, Springer-Verlag, NewYork, 1974.

\bibitem{JS} J. Siran, \textit{Graph Embeddings and Designs}, in: Handbook of Combinatorial Designs. Edited by C. J. Colbourn and J. H. Dinitz. Second edition. Discrete
Mathematics and its Applications, Chapman \& Hall/CRC, Boca Raton, 2007.

\bibitem{White} H. S. White, F. N. Cole, L. D. Cummings, \textit{Complete classification of the
triad systems on fifteen elements}, Memoirs Nat. Acad. Sci. U.S.A. \textbf{14}, 2nd memoir (1919), 1--89.
\end{thebibliography}
\end{document}